\documentclass[a4paper,12pt]{article}     

\usepackage{graphicx}		  
\usepackage{amsmath,amssymb}	

\newtheorem{lemma}{Lemma}[section]
\newtheorem{theorem}[lemma]{Theorem}

\newtheorem{assumption}{Assumption}

\newcommand\LABEL[1]{\label{#1}}

\newenvironment{proof}{
\hspace*{-9mm}
{ \it Proof.}}
{\hfill {$\square$}\vspace{1.5em}}

\def\authorfont{\footnotesize}

\def\ccode#1{\par		
	\vspace*{8pt}
	{\authorfont{\leftskip18pt\rightskip\leftskip
	\noindent #1\par}}\par}

\begin{document}

\begin{center}
{\Large
\bf Minimal Charts of Type $(3,3)$}
\vspace{10pt}
\\ 
{\sc Teruo NAGASE and Akiko SHIMA}
\end{center}

\vspace{10pt}

\begin{abstract}
Let $\Gamma$ be a chart.
For each label $m$, we denote by $\Gamma_m$
the subgraph of $\Gamma$
consisting of all the edges of label $m$ and their vertices. Let $\Gamma$ be a minimal chart of type $(m;3,3)$.
That is, a minimal chart $\Gamma$ has six white vertices,
and both of $\Gamma_m\cap\Gamma_{m+1}$ and $\Gamma_{m+1}\cap\Gamma_{m+2}$ consist of three white vertices.
Then $\Gamma$ is C-move equivalent to a minimal chart
containing a ``subchart'' representing a 2-twist spun trefoil or its ``reflection''.
\end{abstract}

\ccode{2010 Mathematics Subject Classification. Primary 57Q45; Secondary 57Q35.}
\ccode{ {\it Key Words and Phrases}. 2-knot, chart, white vertex. }

\setcounter{section}{0}
\section{{\large Introduction}}


\baselineskip 14pt

Charts are oriented and labeled graphs
in a disk 
with three kinds of vertices
called black vertices, crossings,
and white vertices 
(see Section~\ref{s:Prel} for the precise definition of charts).
Charts correspond to surface braids 
(see \cite[Chapter 14]{BraidBook} for the definition of surface braids).
The closures of surface braids are  oriented closed surfaces embedded in 4-space ${\Bbb R}^4$
 (see \cite[Chapter 23]{BraidBook}). 
A {\it C-move} 
is a local modification between two charts
in a disk.
A C-move induces 
an ambient isotopy between the closures
of the corresponding two surface braids.
Two charts are said to be {\it C-move equivalent} 
if there exists
a finite sequence of C-moves 
which modifies one of the two charts 
to the other.
We will work in the PL category or smooth category. 
All submanifolds are assumed to be locally flat.

A chart is called a {\it ribbon chart}
if the chart is C-move equivalent to 
a chart without white vertices \cite{BraidThree}.
The 4-chart as shown in Fig.~\ref{figTrefoil}(a) 
represents a 2-twist spun trefoil.
It is well known that 
the 2-knot is not a ribbon 2-knot.  
On the other hand, Hasegawa showed that 
if a non-ribbon chart representing a 2-knot 
is minimal, 
then the chart must possess at least six white vertices \cite{H1} 
where a {\it minimal} chart $\Gamma$ means 
its complexity $(w(\Gamma), -f(\Gamma))$ 
is minimal among 
the charts C-move equivalent to 
the chart $\Gamma$ with respect to 
the lexicographic order 
of pairs of integers, 
$w(\Gamma)$ is the number of white vertices 
in $\Gamma$, and 
$f(\Gamma)$ is the number of free edges 
in $\Gamma$.
Nagase, Ochiai, and Shima showed that
there does not exist a minimal chart 
with exactly five white vertices \cite{ONS}.
Nagase and Shima show that
there does not exist a minimal chart
with exactly seven white vertices \cite{ChartApp1}, \cite{ChartAppII},\cite{ChartAppIII}, \cite{ChartApp4}.
Ishida, Nagase, and Shima showed that
any minimal chart with exactly four white vertices
 is C-move equivalent to a chart in two kinds of classes \cite{INS}.

\begin{figure}[ht]
\begin{center}
\includegraphics{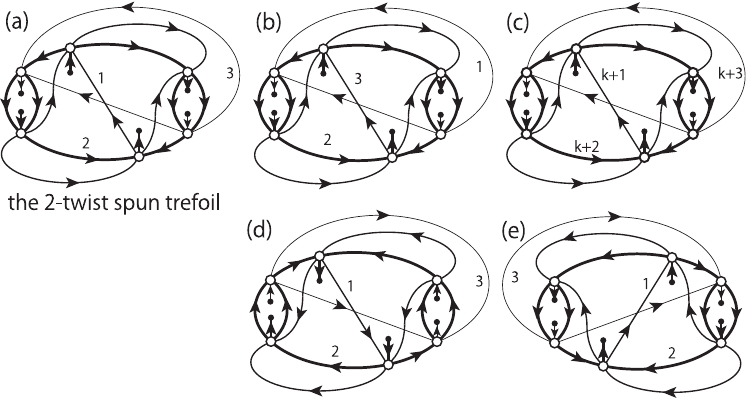}
\end{center}
\caption{\label{figTrefoil} Charts contained in the lor-family of the 4-chart describing the 2-twist spun trefoil. Here $k$ is 
a positive integer.}
\end{figure}

Two charts are said to be 
{\it lor-equivalent}
(Label-Orientation-Reflection equivalent)
provided that
one of the charts is obtained from
the other by a finite sequence of 
the following five modifications:
\begin{enumerate}
\item[(i)]
consider an $n$-chart as an $(n+1)$-chart,
\item[(ii)]
for an $n$-chart,
change all the edges of label $k$ to 
ones of label $n-k$
for each $k=1,2,\cdots,n-1$,
simultaneously 
(see Fig.~\ref{figTrefoil}(b)),
\item[(iii)] 
add a positive constant integer $k$ 
to all the labels simultaneously 
(so that 
the $n$-chart changes to an $(n+k)$-chart)
(see Fig.~\ref{figTrefoil}(c)),
\item[(iv)]
reverse the orientation of 
all the edges (see Fig.~\ref{figTrefoil}(d)),
\item[(v)]
change a chart 
by the reflection in the sphere 
(see Fig.~\ref{figTrefoil}(e)).
\end{enumerate} 
The set of all the charts lor-equivalent to 
a chart $\Gamma$ is called 
the {\it lor-family of $\Gamma$} \cite{Gambit}.
For example, the lor-family of 
the 4-chart describing the 2-twist spun trefoil 
contains the charts as shown in 
Fig.~\ref{figTrefoil}.

Let $\Gamma$ be a chart.
For each label $m$, we denote by $\Gamma_m$
the subgraph of $\Gamma$
consisting of all the edges of label $m$ and their vertices.

Now we define a type of a chart:
Let $\Gamma$ be a chart, 
and $n_1,n_2,\dots,n_p$ integers.
The chart $\Gamma$ is of {\it type $(n_1,n_2,\dots,n_k)$} if there exists a label $m$ of $\Gamma$ satisfying the following three conditions:
\begin{enumerate}
	\item[(i)] For each $i=1,2,\dots, k$, 
	the chart $\Gamma$ contains exactly $n_{i}$ white vertices in $\Gamma_{m+i-1}\cap \Gamma_{m+i}$.
	\item[(ii)] If $i<0$ or $i>k$, then $\Gamma_{m+i}$ does not contain any white vertices.
	\item[(iii)] Both of the two subgraphs $\Gamma_m$ and $\Gamma_{m+k}$ contain at least one white vertex.
\end{enumerate}
If we want to emphasize the label $m$,
then we say that $\Gamma$ is of {\it type $(m;n_1,n_2,\dots,n_k)$}. 
Note that $n_1\ge1$ and $n_k\ge1$ by the condition (iii).

In this paper
we shall show the following theorem:

\begin{theorem}
\LABEL{Type(3,3)} 
Let $\Gamma$ be a minimal chart of type $(3,3)$.
Then $\Gamma$ is C-move equivalent to a minimal chart
containing a subchart in the lor-family of the $2$-twist spun trefoil.
\end{theorem}

The paper is organized as follows.
In Section~\ref{s:Prel},
we define charts.
In Section~\ref{s:LensLoop},
we review a useful lemma for a lens,
a disk whose boundary consists of edges of label $m$ and edges of label $m+1$
satisfying some condition.
And we review a useful lemma for a loop.
In Section~\ref{s:kAngledDisk},
 we review a $k$-angled disk, a disk whose boundary consists of edges of label $m$
and contains exactly $k$ white vertices.
In particular,
we prove a lemma for a 2-angled disk.
In Section~\ref{s:ThreeAngledDisk},
we review three lemmata for a 3-angled disk.
In Section~\ref{s:IOC},
we introduce some property of charts.
In Section~\ref{s:M3M4move},
we prove a useful lemma called New Disk Lemma.
And we give  generalizations of a C-I-M3 move
and a C-I-M4 move.
In Section~\ref{s:NoMinimalSix},
we give four examples of charts of type $(3,3)$.
Three of them are non-minimal charts,
the last one is a minimal chart C-move equivalent to a chart containing a subchart 
in the lor-family of the 2-twist spun trefoil.
In Section~\ref{s:MainTheorem},
we prove Main Theorem (Theorem~\ref{Type(3,3)}).



\section{{\large Preliminaries}}
\label{s:Prel}

Let $n$ be a positive integer. An {\it $n$-chart} is an oriented and labeled graph in a disk,
which may be empty or have closed edges without vertices, called {\it hoops},
satisfying the following four conditions:
\begin{enumerate}
	\item[(i)] Every vertex has degree $1$, $4$, or $6$.
	\item[(ii)] The labels of edges are in $\{1,2,\dots,n-1\}$.
	\item[(iii)] In a small neighborhood of each vertex of degree $6$,
	there are six short arcs, three consecutive arcs are oriented inward and the other three are outward, and these six are labeled $i$ and $i+1$ alternately for some $i$,
	where the orientation and the label of each arc are inherited from the edge containing the arc.
	\item[(iv)] For each vertex of degree $4$, diagonal edges have the same label and are oriented coherently, and the labels $i$ and $j$ of the diagonals satisfy $|i-j|>1$.
\end{enumerate}
We call a vertex of degree $1$ a {\it black vertex,} a vertex of degree $4$ a {\it crossing}, and a vertex of degree $6$ a {\it white vertex} respectively (see Fig.~\ref{fig01}).
Among six short arcs
in a small neighborhood of
a white vertex,
a central arc of each three consecutive arcs
oriented inward or outward 
is called 
a {\it middle arc} at the white vertex
(see Fig.~\ref{fig01}(c)).
There are two middle arcs
in a small neighborhood of
each white vertex.

\begin{figure}[ht]
\centerline{\includegraphics{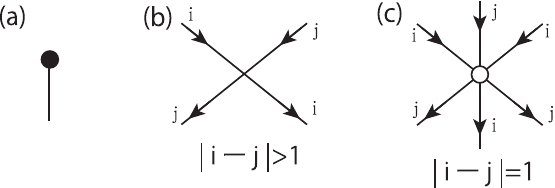}}
\vspace{5mm}
\caption{\label{fig01}(a) A black vertex. (b) A crossing. (c) A white vertex.}
\end{figure}

Let $D_1^2,D_2^2$ be disks,
and $pr_2:D_1^2\times D_2^2\to D_2^2$
the projection defined by $pr_2(x,y)=y$.
Let $Q_n$ be a set of $n$ interior points of $D_1^2$.
A {\it surface braid} $S$ is an oriented surface 
embedded properly in $D_1^2\times D_2^2$
such that
the map $pr_2|_S:S\to D_2^2$ is a branched covering of degree $n$
and $\partial S=Q_n\times\partial D_2^2$
\cite[Chapter 14]{BraidBook}.
A surface braid can be represented by a motion picture method,
a one-parameter family of geometric $n$-braids $\{b_t\}_{t\in[0,1]}$
except for a finite number of values $t_1,t_2,\cdots,t_m\in [0,1]$.
A motion picture for a white vertex is a motion picture as shown in Fig.~\ref{MotionPicture}(a) (cf. \cite[p. 132, Figure 18.5]{BraidBook}).
A motion picture for a crossing is a motion picture as shown in Fig.~\ref{MotionPicture}(b) (cf. \cite[p. 131, Figure 18.4]{BraidBook}).
A motion picture for a black vertex is a motion picture as shown in Fig.~\ref{MotionPicture}(c) (cf. \cite[p. 134, Figure 18.7]{BraidBook}).
A black vertex is corresponding to a singular point of a branched covering map.

\begin{figure}[thb]
\centerline{\includegraphics{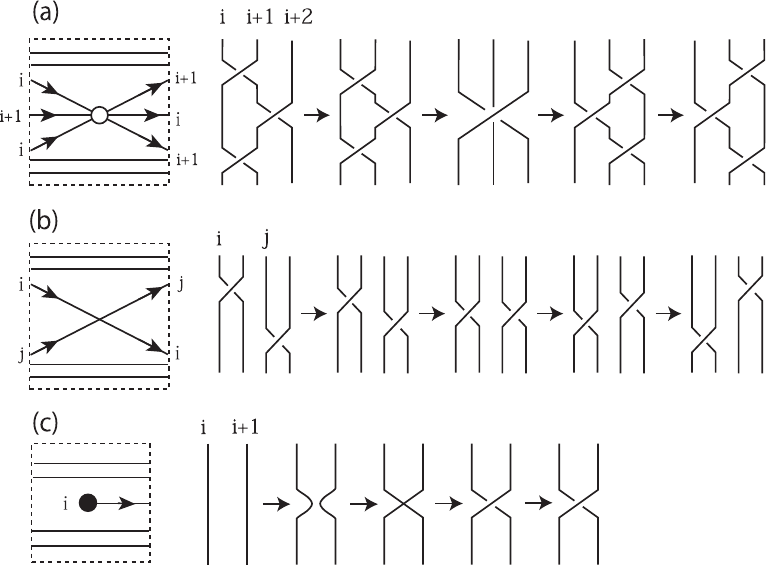}}
\vspace{5mm}
\caption{\label{MotionPicture}}
\end{figure}

Now {\it $C$-moves} are local modifications 
of charts in a disk
as shown in 
Fig.~\ref{fig02}
(cf. \cite{KnottedSurfaces}, \cite{BraidBook}, \cite{Tanaka}).
These C-moves as shown in 
Fig.~\ref{fig02} 
are examples of C-moves.
Note that any C-move is realized by C-moves as shown in 
Fig.~\ref{fig02} and ambient isotopic deformations of the disk (see \cite{KnottedSurfaces}, \cite{BraidBook}, \cite{Tanaka}).


\begin{figure}[ht]
\centerline{\includegraphics{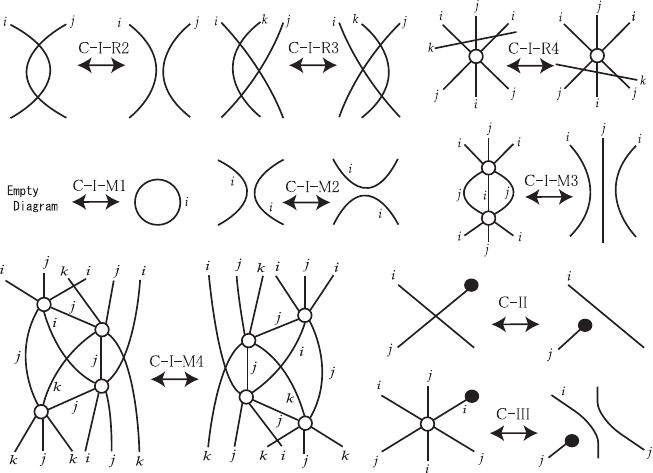}}
\vspace{5mm}
\caption{\label{fig02}For the C-III move, the edge containing the black vertex does not contain a middle arc in the left figure.}
\end{figure}

We showed the difference of a chart in a disk and in a 2-sphere (see \cite[Lemma 2.1]{ChartApp1}).
This lemma follows from that there exists a natural one-to-one correspondence between $\{$charts in $S^2\}/$C-moves and $\{$charts in $D^2\}/$C-moves, conjugations
(\cite[Chapter 23 and Chapter 25]{BraidBook}).
To make the argument simple, we assume that 
the charts lie on the 2-sphere instead of the disk.
\begin{assumption}
In this paper,
all charts are contained in the $2$-sphere $S^2$.
\end{assumption}
We have the special point in the 2-sphere $S^2$, called the point at infinity,
 denoted by $\infty$.
In this paper, all charts are contained in a disk such that the disk 
does not contain the point at infinity $\infty$.


Let $\Gamma$ be a chart.
For each label $m$, we denote by $\Gamma_m$
the subgraph of $\Gamma$
consisting of all the edges of label $m$ and their vertices. Let $L$ be the closure of a connected component 
of the set obtained by taking out 
all the white vertices from $\Gamma_m$.
If $L$ contains at least one white vertex
but does not contain any black vertex,
then $L$ is called an {\it internal edge of label $m$}.
Hence an internal edge may contain a crossing of $\Gamma$.


Let $\Gamma$ be a chart,
and $m$ a label of $\Gamma$. 
A {\it hoop} is a closed edge of $\Gamma$ without vertices 
(hence without crossings, neither).
A {\it ring} is a simple closed curve in $\Gamma_m$ 
containing a crossing but not containing any white vertices. 
An edge of $\Gamma$
is called 
a {\it free edge}
if it has
two black vertices.
An edge of $\Gamma$ is called 
a {\it terminal edge}
if it has
a white vertex and a black vertex.
An internal edge of label $m$
is called a {\it loop} 
if it has only one white vertex.

A hoop or a ring is said to be {\it simple} 
if one of the two complementary domains
of the curve
does not contain any white vertices.

We can assume that
all minimal charts $\Gamma$
satisfy the following four conditions 
(see \cite{ChartApp1},\cite{ChartAppII},\cite{ChartAppIII}):

\begin{assumption}
\label{NoTerminal}
If an edge of $\Gamma$
contains a black vertex,
then the edge is a free edge or a terminal edge.
Hence 
any terminal edge contains a middle arc.
\end{assumption}

\begin{assumption}
\label{NoSimpleHoop}
All free edges and simple hoops in $\Gamma$ 
are moved into a small neighborhood $U_\infty$ 
of the point at infinity $\infty$. 
Hence
we assume that 
$\Gamma$ does not contain free edges
nor simple hoops, 
otherwise mentioned. 
\end{assumption}

\begin{assumption}
\label{AssumptionRing}
Each complementary domain of
any ring and hoop must contain 
at least one white vertex. 
\end{assumption}

\begin{assumption}
The point at infinity $\infty$ is moved in any complementary domain of $\Gamma$.
\end{assumption}

In this paper
for a set $X$ in a space
we denote 
the interior of $X$,
the boundary of $X$ and
the closure of $X$
by Int$X$, $\partial X$
and $Cl(X)$
respectively.


\section{{\large Lenses and loops}}
\label{s:LensLoop}

Let $\Gamma$ be a chart. 
Let $D$ be a disk 
such that 
\begin{enumerate}
\item[(1)] $\partial D$ consists of an internal edge $e_1$ of label $m$ and an internal edge $e_2$ of label ${m+1}$, and 
\item[(2)] any edge containing a white vertex in $e_1$ does not intersect the open disk Int$D$.
\end{enumerate}
Note that $\partial D$ may contain crossings.
Let $w_1$ and $w_2$ be the white vertices in $e_1$. 
If the disk $D$ satisfies one of the following conditions, then $D$ is called  {\it a lens of type $(m,m+1)$}
(see Fig.~\ref{lens}):
\begin{enumerate}
	\item[(i)] Neither $e_1$ nor $e_2$ contains a middle arc. 
	\item[(ii)] One of the two internal edges $e_1$ and $e_2$ contains middle arcs at both white vertices $w_1$ and $w_2$ simultaneously.
\end{enumerate}

\begin{figure}[ht]
\centerline{\includegraphics{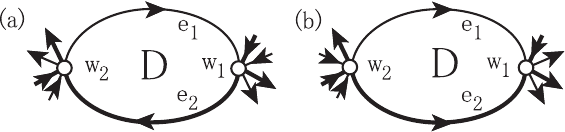}}
\vspace{5mm}
\caption{\label{lens}}
\end{figure}

\begin{lemma}
{\em $($\cite[Theorem 1.1]{ChartApp1} and \cite[Corollary 1.3]{ChartAppII}$)$}
\LABEL{PolandLemma}
Let $\Gamma$ be a minimal chart.
Then we have the following:
\begin{enumerate}
\item[{\rm (a)}] There exist at least three white vertices in the interior of any lens.
\item[{\rm (b)}] If $\Gamma$ contains at most seven white vertices, then there is no lens of $\Gamma$.
\end{enumerate}
\end{lemma}

Let $X$ be a set in a chart $\Gamma$.
Let
 $$w(X)=\text{the number of white vertices in $X$,}$$
$$c(X)=\text{the number of crossings in }X.$$

Let $\ell$ be a loop of label $m$ in a chart $\Gamma$. 
Let $e$ be the edge of label $m$
containing the white vertex in $\ell$ with $e\not\subset\ell$. 
Then the loop $\ell$ bounds two disks on the 2-sphere. 
One of the two disks does not contain the edge $e$. 
The disk is called {\it the associated disk of the loop $\ell$} 
(see Fig.~\ref{fig03}).


\begin{figure}[ht]
\centerline{\includegraphics{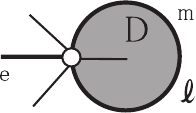}}
\vspace{5mm}
\caption{\label{fig03}
The gray region is the associated disk of a loop $\ell$.}
\end{figure}

\begin{lemma}
\LABEL{LoopTwoVertices}
{\em $($\cite[Lemma 4.2]{ChartAppII}$)$}
Let $\Gamma$ be a minimal chart with a loop $\ell$ of label $m$ with the associated disk $D$.
Then $w(\Gamma\cap${\rm Int}$D)\ge2$ and 
$w(\Gamma\cap(S^2-D))\ge2$.
\end{lemma}

\begin{lemma}
\LABEL{LoopSixVertices}
{\em $($cf. \cite[Theorem 1.4]{ChartAppII}$)$}
Let $\Gamma$ be a minimal chart with $w(\Gamma)=6$.
If $\Gamma$ contains a loop,
then $\Gamma$ is of type $(2,4)$ or $(4,2)$.
\end{lemma}

 In our argument  we often construct a chart $\Gamma$. 
On the construction of a chart $\Gamma$, 
for a white vertex $w$,  
if $w$ contains the two internal edges of label $m$ 
and one terminal edge of label $m$ 
(see Fig.~\ref{MarkVertex}(a) and (b)), 
then we remove the terminal edge and
put a black dot at the center of the white vertex 
as shown in Fig.~\ref{MarkVertex}(c).


\begin{figure}[ht]
\centerline{\includegraphics{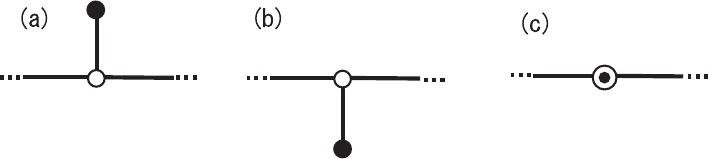}}
\vspace{5mm}
\caption{\label{MarkVertex}}
\end{figure}

\begin{lemma}
{\rm (cf. \cite[Lemma 6.1 and Lemma 6.2]{ChartAppIII})}
\LABEL{TwoThreeLoop}
Let $\Gamma$ be a minimal chart.
Let $G$ be a connected component of $\Gamma_m$.
\begin{enumerate} 
\item[{\rm (a)}] If $w(G)\ge1$, then $w(G)\ge2$.
\item[{\rm (b)}] If $w(G)=3$ and $G$ does not contain any loop,
then $G$ is the subgraph as shown 
in Fig.~\ref{figThree}. 
\end{enumerate} 
\end{lemma}


\begin{figure}[ht]
\centerline{\includegraphics{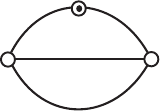}}
\vspace{5mm}
\caption{\label{figThree}}
\end{figure}

{\bf Notation.} We use the following notation:

In our argument,
we often need a name for an unnamed edge by using a given edge and a given white vertex.
For the convenience,
we use the following naming:
Let $e',e_i,e''$ be three consecutive internal edges (or edges) containing a white vertex $w_j$. Here, 
the two edges $e'$ and $e''$ are unnamed edges. 
There are six arcs in a neighborhood $U$ of the white vertex $w_j$. 
If the three arcs $e'\cap U$, $e_i \cap U$, $e'' \cap U$ lie anticlockwise around the white vertex $w_j$ in this order, 
then $e'$ and $e''$ are denoted by $a_{ij}$ and $b_{ij}$ 
respectively (see Fig.~\ref{BeforeAfter}).
There is a possibility $a_{ij}=b_{ij}$ if they are contained in a loop.

\begin{figure}[ht]
\centerline{\includegraphics{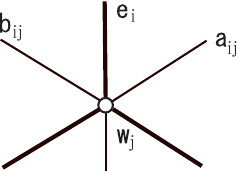}}
\vspace{5mm}
\caption{\label{BeforeAfter}}
\end{figure}



\section{{\large $k$-angled disks}}
\label{s:kAngledDisk}

Let $\Gamma$ be a chart, $m$ a label of $\Gamma$, $D$ a disk, and $k$ a positive integer.
If $\partial D\subset \Gamma_m$
and if $\partial D$ contains exactly $k$ white vertex, 
then $D$ is called {\it a $k$-angled disk of $\Gamma_m$}. 
Note that the boundary $\partial D$ may contain crossings, 
and
each of two disks bounded by a loop of label $m$ is a $1$-angled disk of $\Gamma_m$.

Let $\Gamma$ be a chart, and
$m$ a label of $\Gamma$.
Let $e$ be an internal edge (or a ternimal edge) of label $m$.
The edge $e$ is called a {\it feeler} of a $k$-angled disk $D$ of $\Gamma_m$
if the edge $e$ intersects $N-\partial D$
where $N$ is a regular neighborhood of $\partial D$ in $D$.

Let $\Gamma$ be a chart. If an object consists of some edges of $\Gamma$, arcs in edges of $\Gamma$ and arcs around white vertices,
then the object is called {\it a pseudo chart}.

\begin{lemma}
{\em $($\cite[Lemma 5.2]{ChartAppII} and 
\cite[Theorem 1.1]{ChartAppIII}$)$}
\LABEL{TwoAngledDiskWithoutFeeler}
Let $\Gamma$ be a minimal chart, and $m$ a label of $\Gamma$.
Let  $D$ be a $2$-angled disk of $\Gamma_m$ without feelers. 
If $w(\Gamma\cap${\rm Int}$D)=0$,
then a regular neighborhood of $D$ contains the pseudo chart 
as shown in Fig.~\ref{fig2AngledDisk}{\rm (a)}.
\end{lemma}


\begin{figure}[ht]
\centerline{\includegraphics{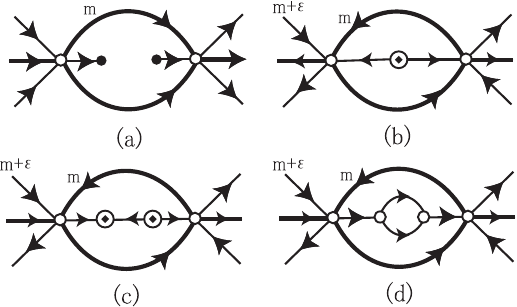}}
\vspace{5mm}
\caption{\label{fig2AngledDisk}
The thick lines are edges of label $m$,
and $\varepsilon\in\{+1,-1\}$. }
\end{figure}

Let $\Gamma$ be a chart, and $m$ a label of $\Gamma$.
Let $D$ be a $k$-angled disk of $\Gamma_m$, 
and 
$G$ a pseudo chart in $D$ with $\partial D\subset G$.
Let $r:D\to D$ be a reflection of $D$, and $G^*$ the pseudo chart obtained from $G$ by changing the orientations of all of the edges.
Then the set $\{G,G^*, r(G), r(G^*)\}$ 
is called the {\it RO-family of the pseudo chart $G$}.

{\it Warning.}
To draw a pseudo chart in a RO-family,
we draw a part of $\Gamma\cap N$
here $N$ is a regular neighborhood of a $k$-angled disk $D$.

\begin{lemma}
{\em $($cf. \cite[Lemma 5.3, Lemma 5.4 and Lemma 5.5]{ChartAppII}$)$}
\LABEL{TwoAngledDisk}
Let $\Gamma$ be a minimal chart, and $m$ a label of $\Gamma$.
Let  $D$ be a $2$-angled disk of $\Gamma_m$ without feelers. Let $w_1$ be a white vertex in $\partial D$, and 
$\varepsilon$ the integer in $\{+1,-1\}$ 
with $w_1\in\Gamma_{m+\varepsilon}$.
Suppose that $\partial D$ is oriented clockwise or anticlockwise.
Then we have the following:
\begin{enumerate}
\item[{\rm (a)}] $w(\Gamma\cap${\rm Int}$D)\ge1$. 
\item[{\rm (b)}] If $w(\Gamma\cap${\rm Int}$D)=1$,
then $D$ contains an element of the RO-family of the pseudo chart as shown in 
Fig.~\ref{fig2AngledDisk}{\rm (b)}.
\item[{\rm (c)}] 
If $w(\Gamma\cap${\rm Int}$D)=2$ and 
$w(\Gamma_{m+\varepsilon}\cap${\rm Int}$D)=2$, 
then $D$ contains an element of RO-families of the two pseudo charts as shown in Fig.~\ref{fig2AngledDisk}{\rm (c)} and {\rm (d)}.
\end{enumerate}
\end{lemma}

\begin{proof}
Let $w_2$ be the white vertex in $\partial D$ different from $w_1$,
and $\delta$ the integer in $\{+1,-1\}$ 
with $w_2\in\Gamma_{m+\delta}$.
Let $e_1$ (resp. $e_2$) be an internal edge (or a terminal edge) label ${m+\varepsilon}$ (resp. $m+\delta$) in $D$
with $e_1\ni w_1, e_2\ni w_2$.

Suppose $w(\Gamma\cap${\rm Int}$D)=0$.
Since
 $\partial D$ is oriented clockwise or anticlockwise,
the edge $e_1$ is not middle at $w_1$.
Thus the edge $e_1$ is not a terminal edge
by Assumption~\ref{NoTerminal}.
Since $w(\Gamma\cap{\rm Int}D)=0$,
we have $e_1\ni w_2$,
i.e. $e_1=e_2$.
Hence the edge $e_1$ separates the disk $D$ into two lenses
each of which does not contain any white vertices 
in its interior.
This contradicts Lemma~\ref{PolandLemma}(a). 
Hence $w(\Gamma\cap${\rm Int}$D)\ge1$.
Thus Statement (a) holds.

Suppose $w(\Gamma\cap${\rm Int}$D)=1$.
Similarly
we can show that the edge $e_1$ is not a terminal edge 
and $e_1\not\ni w_2$.
Hence $e_1$ contains a white vertex in Int$D$,
say $w_3$.
Similarly we can show that 
the edge $e_2$ is not a terminal edge 
and $e_2\not\ni w_1$.
Thus $e_2$ contains the white vertex $w_3$ in Int$D$.
Hence there exists a terminal edge of label $m+\varepsilon$
containing $w_3$.
Thus the disk $D$ contains  the pseudo chart as shown in Fig.~\ref{fig2AngledDisk}(b).
Hence Statement (b) holds.

Suppose $w(\Gamma\cap${\rm Int}$D)=2$
and $w(\Gamma_{m+\varepsilon}\cap${\rm Int}$D)=2$.
Similarly
we can show that 
neither $e_1$ nor $e_2$ is a terminal edge, 
and $e_1\not\ni w_2$, $e_2\not\ni w_1$.
Let $w_3$ be the white vertex in ${\rm Int}D$
with $e_1\ni w_3$.
Let $w_4$ be the white vertex in ${\rm Int}D$
different from $w_3$.

If $e_2\ni w_3$,
then there exists a loop of label $m+\varepsilon$ 
containing $w_4$
 bounding a disk $E$ and
 Int$E$ does not contains any white vertices.
This contradicts Lemma~\ref{LoopTwoVertices}.
Thus $e_2\not\ni w_3$.
Since $e_2$ is not a terminal edge and since $e_2\not\ni w_1$,
we have $e_2\ni w_4$.

There are three cases:
\begin{enumerate}
\item[(1)] both of $w_3,w_4$ are contained in loops of label $m+\varepsilon$,
\item[(2)] there exists an internal edge $e_3$ of label ${m+\varepsilon}$ 
containing $w_3$ and $w_4$,
and both of $w_3,w_4$ are contained in terminal edges of label ${m+\varepsilon}$,
\item[(3)] there exist two internal edges $e_3,e_4$ of label ${m+\varepsilon}$ 
containing $w_3$ and $w_4$.
\end{enumerate}

Without loss of generality 
we can assume that 
$e_1$ is oriented from $w_1$ to $w_3$.

For {\bf Case (1)},
the associated disk of each loops
does not contain any white vertices in its interior.
This contradicts Lemma~\ref{LoopTwoVertices}.
Hence Case (1) does not occur.

For {\bf Case (2)},
the edge $e_3$ is oriented from $w_4$ to $w_3$
and the edge $e_2$ is oriented from $w_4$ to $w_2$.
Hence $D$ contains the pseudo chart 
as shown in Fig.~\ref{fig2AngledDisk}(c).

For {\bf Case (3)},
there exists a $2$-angled disk $E$ of $\Gamma_{m+\varepsilon}$ in $D$ with $\partial E=e_3\cup e_4$.
Since $E$ has no feelers,
by Lemma~\ref{TwoAngledDiskWithoutFeeler}
we have that 
either the two edges $e_3,e_4$ are oriented from 
$w_3$ to $w_4$, or
the two edges $e_3,e_4$ are oriented from 
$w_4$ to $w_3$.
Since $e_1$ is oriented from $w_1$ to $w_3$,
the two edges $e_3,e_4$ are oriented from 
$w_3$ to $w_4$.
Hence $D$ contains the pseudo chart 
as shown in Fig.~\ref{fig2AngledDisk}(d).
Thus Statement (c) holds.
\end{proof}


\section{{\large $3$-angled disks}}
\label{s:ThreeAngledDisk}

\begin{lemma}
{\rm (\cite[Lemma 4.3 and Lemma 4.5]{ChartAppIII})}
\LABEL{3AngledDiskWithoutFeeler}
 Let $\Gamma$ be a minimal chart.
Let $D$ be a $3$-angled disk of $\Gamma_m$
without feelers.
If $w(\Gamma\cap${\rm Int}$D)=0$,
then $D$ contains an element in the RO-families of 
the two pseudo charts as shown in Fig.~\ref{fig3AngledDisk}{\rm (a)} and {\rm (b)}.
\end{lemma}


\begin{figure}[ht]
\centerline{\includegraphics{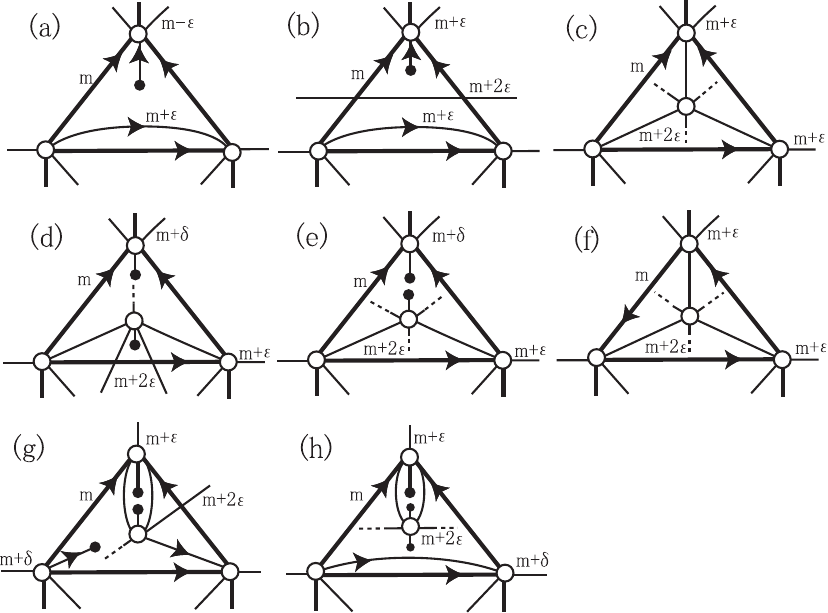}}
\vspace{5mm}
\caption{\label{fig3AngledDisk}The 3-angled disks (g) and (h) have one feeler, the others do not have any feelers.}
\end{figure}

Let $\Gamma $ and $\Gamma^\prime $ be C-move equivalent charts. 
Suppose that a pseudo chart $X$ of $\Gamma$ is also a pseudo chart of $\Gamma^\prime$. 
Then we say that 
$\Gamma$ is modified to $\Gamma^\prime$ by {\it C-moves keeping $X$ fixed}.
In Fig.~\ref{fig04},
we give examples of C-moves keeping pseudo charts  fixed.

\begin{figure}[ht]
\centerline{\includegraphics{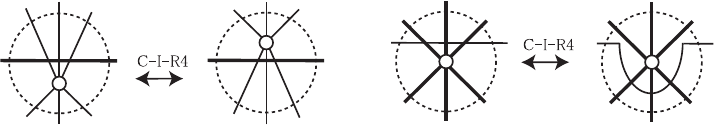}}
\vspace{5mm}
\caption{\label{fig04} C-moves keeping thicken figures fixed.}
\end{figure}

Let $\Gamma$ be a minimal chart,
and $m$ a label of $\Gamma$.
Let $D$ be a $k$-angled disk of $\Gamma_m$.
The pair of integers 
$(w(\Gamma\cap$Int$D),c(\Gamma\cap\partial D))$
is called the {\it local complexity 
with respect to $D$}.
Let ${\Bbb S}$ be the set of all
minimal charts obtained from $\Gamma$ by C-moves in a regular neighborhood of
$D$ keeping $\partial D$ fixed.
The chart $\Gamma$ is said to be 
{\it locally minimal
with respect to $D$}
if its local complexity
with respect to $D$
is minimal
among the charts in ${\Bbb S}$
with respect to 
the lexicographic order of pairs of integers.

Now 
for a chart $\Gamma$,
a $k$-angled disk $D$ of $\Gamma_m$ is {\it special} 
provided that
any feeler is a terminal edge
where a feeler is an internal edge (or a terminal edge) of label $m$
intersecting $\partial D$ and Int$D$.

\begin{lemma}
{\rm (\cite[Theorem 1.2]{ChartAppIII})}
\LABEL{Theorem3AngledDisk}
 Let $\Gamma$ be a minimal chart.
Let $D$ be a special $3$-angled disk of $\Gamma_m$
such that $\Gamma$ is locally minimal
with respect to $D$.
If $w(\Gamma\cap${\rm Int}$D)\le1$,
then $D$ contains an element in the RO-families of the eight pseudo charts as shown in Fig.~\ref{fig3AngledDisk}.
\end{lemma}

\begin{lemma}$($Triangle Lemma$)$
{\rm (\cite[Lemma 8.3]{ChartApp4})}
\LABEL{LemmaTriangle}
For a minimal chart $\Gamma$, 
if there exists a $3$-angled disk $D_1$ of $\Gamma_m$ without feelers in a disk $D$ as shown in Fig.~\ref{figTriangleLemma},
then $w(\Gamma\cap${\rm Int}$D_1)\ge1$.
\end{lemma}

\begin{figure}[ht]
\centerline{\includegraphics{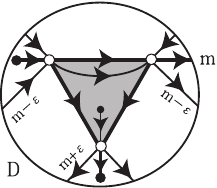}}
\vspace{5mm}
\caption{\label{figTriangleLemma}
The gray region is a 3-angled disks $D_1$. 
The thick lines are edges of label $m$,
and $\varepsilon\in\{+1,-1\}$.}
\end{figure}


\section{{\large IO-Calculation}}
\label{s:IOC}

Let $e$ be an edge or a simple arc.
Suppose that $e$ is not a loop, a hoop nor a ring.
Let $v,v'$ be the two end points of $e$.
Then we denote
$$\partial e=\{v,v'\}, \ \ {\rm Int}e=e-\partial e.$$

Let $\Gamma$ be a chart,
 and $v$ a vertex. 
Let $\alpha$ be a short arc of $\Gamma$ in a small neighborhood of $v$ with $v\in \partial \alpha$. 
If the arc $\alpha$ is oriented to $v$, then $\alpha$ is called {\it an inward arc}, 
and otherwise $\alpha$ is called {\it an outward arc}. 
 
Let $\Gamma$ be an $n$-chart. 
Let $F$ be a closed domain with $\partial F\subset \Gamma_{k-1}\cup\Gamma_{k}\cup \Gamma_{k+1}$ for some label $k$ of $\Gamma$, where $\Gamma_0=\emptyset$ and $\Gamma_{n}=\emptyset$. 
By Condition (iii) for charts,
in a small neighborhood of each white vertex, there are three inward arcs and three outward arcs.
Also in a small neighborhood of each black vertex, there exists only one inward arc or one outward arc.
We often use the following fact, 
when we fix (inward or outward) arcs 
near white vertices and black vertices: 
\begin{enumerate}
\item[]
{\it The number of inward arcs contained in $F\cap \Gamma_k$ is equal to the number of outward arcs in $F\cap \Gamma_k$.
}
\end{enumerate}
When we use this fact, 
we say that we use {\it IO-Calculation with respect to $\Gamma_k$ in $F$}.
For example, in a minimal chart $\Gamma$ of type $(m;3,3)$, 
consider the pseudo chart as shown in 
Fig.~\ref{figIOC}(a) and
\begin{enumerate} 
\item[(1)] $D_1$ is the 2-angled disk of $\Gamma_m$
with $w(\Gamma\cap{\rm Int}D_1)=1$,
\item[(2)] $D_3$ is the 3-angled disk of $\Gamma_m$
with $w(\Gamma\cap{\rm Int}D_3)=0$.
\end{enumerate}
Let $e$ be the internal edge of label $m$ oriented from $w_3$ to $w_2$.
Let $F$ be the 3-angled disk of $\Gamma_{m+1}$ containing $e$ with $\partial F\ni w_2,w_3,w_4$,
and $e_4$ the terminal edge of label $m+1$ containing $w_4$.
Then we can show that
$e_4\not\subset F$.
For if not,
then $e_4\subset F$ (see Fig.~\ref{figIOC}(b)).
Hence in $F$
there are two edges of label ${m+2}$ 
containing $w_4$
but not middle at $w_4$.
By Assumption~\ref{NoTerminal}
neither two edges are terminal edges.
Since $w(\Gamma\cap{\rm Int}F)=0$,
the number of inward arcs in $F\cap \Gamma_{m+2}$ is two,  
but the number of outward arcs in $F\cap \Gamma_{m+2}$ is zero. 
This is a contradiction. 
Instead of the above argument, 
we just say that 
\begin{enumerate}
\item[(3)]
{\it $e_4\not\subset F$
by IO-Calculation with respect to $\Gamma_{m+2}$ in $F$.}
\end{enumerate}

\begin{figure}[ht]
\centerline{\includegraphics{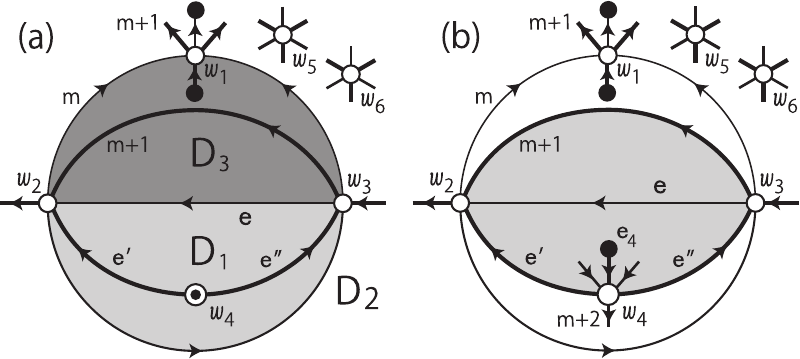}}
\vspace{5mm}
\caption{\label{figIOC} The thick lines are edges of label $m+1$.
(a) The light gray region is the disk $D_1$,
the dark gray region is the disk $D_3$.
(b) The gray region is the disk $F$.}
\end{figure}


\section{{\large C-I-M3 moves and C-I-M4 moves}}
\label{s:M3M4move}

Let $\Gamma$ be a chart, and $D$ a disk.
Let $\alpha$ be a simple arc in $\partial D$,
and $\gamma$ a simple arc in an internal edge of label $k$.
We call the simple arc $\gamma$
a {\it {$(D,\alpha)$-arc}} of label $k$
provided that 
$\partial \gamma \subset $Int$\alpha$
and
Int$\gamma\subset $Int$D$. 
If there is no $(D,\alpha)$-arc in $\Gamma$,
then the chart $\Gamma$ is said to be
$(D,\alpha)$-{\it arc free}.

\begin{lemma}
{\em (cf. \cite[Lemma 3.2]{ChartApp1})} 
$($New Disk Lemma$)$
\LABEL{NewDiskLemma}
Let $\Gamma$ be a chart and
$D$ a disk 
whose interior does not contain 
a white vertex nor a black vertex of $\Gamma$.
Let $\alpha$ be a simple arc in $\partial D$ 
such that ${\rm Int}\alpha$ does not contain 
a white vertex nor a black vertex of $\Gamma$.
Let $V$ be a regular neighborhood of $\alpha$. 
Suppose that the arc $\alpha$ satisfies 
one of the following two conditions:
\begin{enumerate}
\item[{\rm (a)}] 
The arc $\alpha$ is contained in 
an internal edge of some label $k$ of $\Gamma$.
\item[{\rm (b)}] 
$\Gamma\cap\partial \alpha=\emptyset$
and 
if an edge of $\Gamma$ intersects the arc $\alpha$,
then the edge transversely intersects the arc $\alpha$.
\end{enumerate}
Then by applying C-I-M2 moves, C-I-R2 moves, 
and C-I-R3 moves in $V$, 
there exists 
a $(D,\alpha)$-arc free chart $\Gamma'$ 
obtained from the chart $\Gamma$ 
keeping $\alpha$ fixed 
$($see Fig.~\ref{figNewDiskLemma}$)$.
\end{lemma}

\begin{proof}
{\bf Case (a)}. 
We prove by induction on the number of $(D,\alpha)$-arcs.
Let $n$ be the number of $(D,\alpha)$-arcs. 
If $n=0$, then $\Gamma$ is $(D,\alpha)$-arc free.
Suppose $n>0$.
Then there exists a $(D,\alpha)$-arc $L$ such that
the disk $D_L$ bounded by $L$ and 
an arc $L_\alpha$ in $\alpha$ 
contains no other $(D,\alpha)$-arc. 
Now for a proper arc of $D_L$ if the arc is in an internal edge 
and transversely intersects Int$L_\alpha$, 
then it must intersect the arc $L$.
Let $s$ be the number of such arcs 
transversely intersecting Int$L_\alpha$. 

Let $h$ be the label of $L$, and 
$\widetilde L$ the connected component of 
$\Gamma_h\cap(D\cup V)$ containing the arc $L$. 
If $s>0$, then 
by deforming $\widetilde L$ in $V$
by C-I-R2 moves and C-I-R3 moves, 
we can push an end point of $L$ 
near the other end point of $L$ 
along $L_\alpha$ 
(see Fig.~\ref{figNewDiskLemma})
so that we can assume that $s=0$. 

By applying a C-I-M2 move and 
a C-I-R2 move, 
we can split the arc $\widetilde L$ 
to a ring (or a hoop) $R$ and 
an arc $L'$ to get a new chart $\Gamma'$ 
with $(R\cup L')\cap \alpha=\emptyset$.
Hence by induction we can assume that 
the chart is $(D,\alpha)$-arc free.

Similarly we can show for {\bf Case (b)}.
\end{proof}

\begin{figure}[ht]
\centerline{\includegraphics{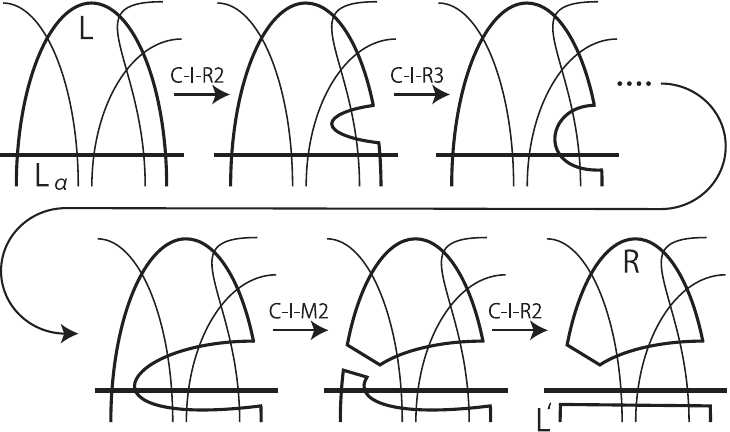}}
\vspace{5mm}
\caption{\label{figNewDiskLemma}}
\end{figure}

\begin{lemma}
{\em (\cite[Lemma 8.2]{ChartAppIII})} 
\LABEL{MM+2Edge}
Let $\Gamma$ be a chart,
$e$ an internal edge of label $m$,
and $w_1,w_2$ the white vertices in $e$.
Suppose $w_1\in\Gamma_{m-1}$ and $w_2\in\Gamma_{m+1}$.
Then for any neighborhood $V$ of the edge $e$, 
there exists a chart $\Gamma'$
obtained from the chart $\Gamma$ by C-I-R2 moves, 
C-I-R3 moves and C-I-R4 moves
in $V$ keeping $\Gamma_{m-1}\cup\Gamma_m\cup\Gamma_{m+1}$ fixed 
such that the edge $e$ does not contain any crossings.
\end{lemma}

Let $\Gamma$ be a chart and 
$k$ a label of $\Gamma$.
If a disk $D$ satisfies the following three conditions,
then $D$ is called an 
{\it M4-disk of label $k$}
 (see Fig.~\ref{figM4-disk}(a)).
\begin{enumerate}
\item[(i)] 
$\partial D$ consists of four internal edges 
$e_1,e_2,e_3,e_4$ of label $k$ 
situated on $\partial D$ 
in this order.
\item[(ii)] 
Set $w_1=e_1\cap e_4,w_2=e_1\cap e_2,
w_3=e_2\cap e_3,w_4=e_3\cap e_4$.
Then
\begin{enumerate}
\item[(a)] 
$D\cap\Gamma_{k-1}$ consists of an internal edge $e_5$ 
connecting $w_1$ and $w_3$, 
and
\item[(b)] 
$D\cap\Gamma_{k+1}$ consists of an internal edge $e_6$ 
connecting $w_2$ and $w_4$. 
\end{enumerate}
\item[(iii)] 
Int$D$ does not contain any white vertex.
\end{enumerate}
We call the union 
$X=\cup_{i=1}^6 e_i$ 
the {\it M4-pseudo chart} for the disk $D$, and 
$(w_1,w_2,w_3,w_4;e_1,e_2,e_3,e_4;e_5,e_6)$
the {\it fundamental information} 
for the M4-disk $D$.

\begin{figure}[ht]
\centerline{\includegraphics{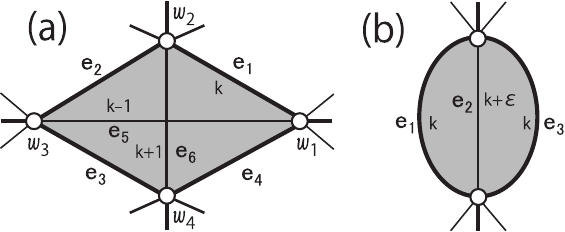}}
\vspace{5mm}
\caption{\label{figM4-disk}
(a) The gray region is the M4-disk. 
(b) The gray region is the M3-disk.}
\end{figure}

\begin{lemma}
\LABEL{M4-diskLemma}
Let $\Gamma$ be a chart, and 
$k$ a label of $\Gamma$. 
Suppose that $D$ is an M4-disk 
of label $k$ 
with an M4-pseudo chart $X$. 
Then by deforming $\Gamma$ 
in a regular neighbourhood of $D$ 
without increasing the complexity of $\Gamma$, 
the chart $\Gamma$ is 
C-move equivalent to a chart $\Gamma'$ 
with
$D\cap(\cup_{i=k-2}^{k+2}\Gamma'_i)=X$.
\end{lemma}

\begin{proof}
Let 
$(w_1,w_2,w_3,w_4;e_1,e_2,e_3,e_4,e_5,e_6)$
be the fundamental information 
for the M4-disk $D$. 
Then 
\begin{enumerate}
\item[(1)]
Int$D$ does not contain any white vertex.
\end{enumerate}

First of all, applying Lemma~\ref{MM+2Edge} 
for $e_1,e_3$ respectively 
keeping $\Gamma_{k-1}\cup\Gamma_{k}\cup\Gamma_{k+1}$  
fixed, 
we can assume that 
\begin{enumerate}
\item[(2)]
neither $e_1$ nor $e_3$ 
contains any crossing.
\end{enumerate}
Secondly, applying New Disk Lemma (Lemma~\ref{NewDiskLemma}) 
for $e_2,e_4\subset \partial D$ respectively, 
the chart $\Gamma$ is C-move equivalent to 
a chart $\Gamma'$ 
deformed in regular neighbourhoods 
of $e_2$ and $e_4$
so that 
\begin{enumerate}
\item[(3)]
the chart $\Gamma'$ is 
$(D,e_2)$-arc free and 
$(D,e_4)$-arc free. 
\end{enumerate}
By Statement (2) and Statement (3), 
for each label $i\not =k$, 
if an internal edge of label $i$ in $\Gamma'$
intersects the edge $e_2$ (resp. $e_4$), 
then it must intersect the edge $e_4$ (resp. $e_2$). 
Since $e_5$ is of label $k-1$ and 
$e_6$ is of label $k+1$, and
since each of $e_5$ and $e_6$ separates 
the two edges $e_2$ and $e_4$ in the disk $D$, 
any internal edge of label ${k-2}$ or ${k+2}$ in $\Gamma'$ 
does not intersect $e_2\cup e_4$.
Hence Statement (1) implies that 
each connected component of 
$D\cap \Gamma_{k-2}'$ and 
$D\cap \Gamma_{k+2}'$ is 
a simple hoop or a simple ring. 
Thus by Assumption~\ref{NoSimpleHoop} and 
Assumption~\ref{AssumptionRing}, 
we have that 
$D\cap (\Gamma'_{k-2}\cup\Gamma'_{k+2})=\emptyset$. 
Again by Assumption~\ref{NoSimpleHoop} and 
Assumption~\ref{AssumptionRing}, 
we can assume that 
$D\cap 
(\Gamma'_{k-1}\cup\Gamma'_{k}\cup\Gamma'_{k+1})$
does not contain a hoop nor a ring. 
Hence we have $D\cap(\cup_{i=k-2}^{k+2}\Gamma'_i)=X$.
Thus we have the result.
\end{proof}

Let $\Gamma$ be a chart, 
$k$ a label of $\Gamma$,
and $\varepsilon\in\{+1,-1\}$.
A disk $D$ is called an 
{\it M3-disk of type $(k,k+\varepsilon)$}
provided that (see Fig.~\ref{figM4-disk}(b))
\begin{enumerate}
\item[(i)] 
$\partial D$ consists of two internal edges $e_1,e_3$ of label $k$,
\item[(ii)] $D\cap \Gamma_k=\partial D$,
\item[(iii)] 
$D\cap\Gamma_{k+\varepsilon}$ consists of 
an internal edge $e_2$ 
with $\partial e_2=\partial e_1$,
\item[(iv)] 
Int$D$ does not contain 
a white vertex nor a black vertex.
\end{enumerate}
We call the union 
$X=e_1\cup e_2\cup e_3$ 
the {\it M3-pseudo chart} for the disk $D$.

Similarly we can show the following lemma.
\begin{lemma}
\LABEL{M3-diskLemma}
Let $\Gamma$ be a chart, and 
$k$ a label of $\Gamma$. 
Suppose that $D$ is an M3-disk
of type $(k,k+\varepsilon)$ 
with an M3-pseudo chart $X$. 
Then by deforming $\Gamma$ 
in a regular neighbourhood of $D$ 
without increasing the complexity of $\Gamma$, 
the chart $\Gamma$ is 
C-move equivalent to a chart $\Gamma'$ 
with
$D\cap(\Gamma'_{k-\varepsilon}\cup\Gamma_k'\cup \Gamma'_{k+\varepsilon}\cup \Gamma'_{k+2\varepsilon})=X$.
\end{lemma}

\section{{\large Non-minimal charts with six white vertices}}
\label{s:NoMinimalSix}

\begin{lemma}
\LABEL{NoMinimalChartSix}
Let $\Gamma$ be a chart with $w(\Gamma)=6$.
If $\Gamma$ contains one of the three subcharts as shown in 
Fig.~\ref{figNoMinimalChartSix}{\rm (a), (b)} and {\rm (c)},
then $\Gamma$ is not minimal.
\end{lemma}

\begin{figure}[ht]
\centerline{\includegraphics{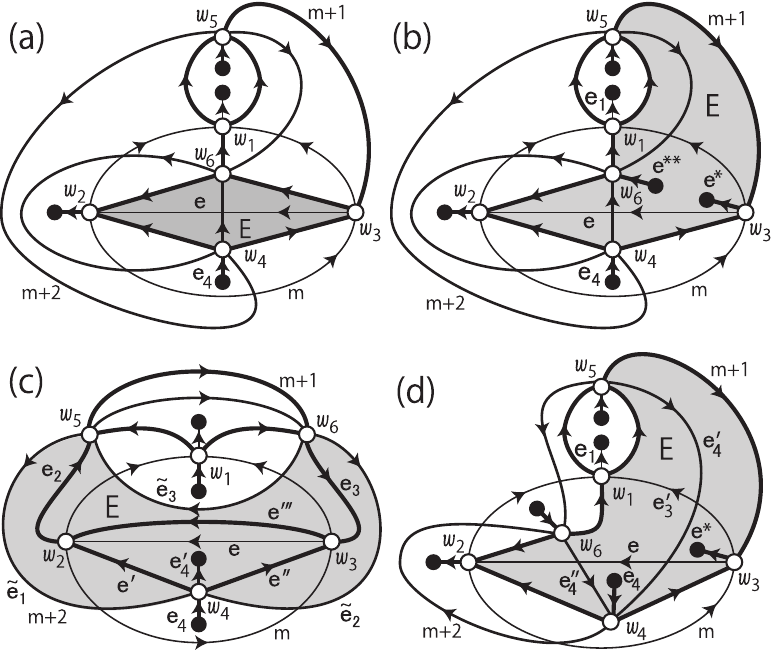}}
\vspace{5mm}
\caption{\label{figNoMinimalChartSix}
The thick lines are edges of label $m+1$.
(a) The gray region is the 4-angled disk $E$.
(b) The gray region is the 6-angled disk $E$.
(c) The gray region is the 3-angled disk $E$.
(d) The gray region is the 6-angled disk $E$.
}
\end{figure}

\begin{proof}
Suppose that $\Gamma$ contains the subchart as shown in 
Fig.~\ref{figNoMinimalChartSix}(a).
We use notations in 
Fig.~\ref{figNoMinimalChartSix}(a).
Let $E$ be the 4-angled disk 
containing the internal edge $e$ with 
$\partial E\ni w_{2},w_{3},w_{4},w_{6}$.
Then the disk is an M4-disk 
of label $m+1$.
Then by Lemma~\ref{M4-diskLemma},
we can apply a C-I-M4 move 
in a neighborhood of $E$,
we obtain a chart 
containing a subchart 
as shown in 
Fig.~\ref{figNoMinimalChartSixProof}(a).
The terminal edge containing $w_2$ 
is not middle at $w_2$.
Applying a C-III move,
we can eliminate the white vertex $w_2$.
Hence $\Gamma$ is not minimal.

Suppose that 
$\Gamma$ contains the subchart as shown in 
Fig.~\ref{figNoMinimalChartSix}(b).
We use notations in 
Fig.~\ref{figNoMinimalChartSix}(b),
here $e^*,e^{**}$ are the terminal edges 
of label $m+1$
containing $w_3,w_6$ respectively, 
$e_1$ the terminal edge of label $m$ 
containing $w_1$.
Let $E$ be the 6-angled disk of $\Gamma_{m+1}$ 
with $E\supset e^*\cup e^{**}$ but $E\not\supset e_1$, 
shown in Fig.~\ref{figNoMinimalChartSix}(b).
By Assumption~\ref{NoSimpleHoop} 
and Assumption~\ref{AssumptionRing},
we can assume that 
$E$ does not contain 
any free edges, hoops nor rings of label $m,m+1,m+2$.
Let $\gamma$ be a simple arc
connecting two black vertices in $e^*$ and $e^{**}$
with
$(\Gamma_m\cup\Gamma_{m+1}\cup\Gamma_{m+2})
\cap{\rm Int}\gamma=\emptyset$.
Applying C-II moves and a C-I-M2 move along $\gamma$,
we obtain a new free edge of label $m+1$.
Hence $\Gamma$ is not minimal.

Suppose that 
$\Gamma$ contains the subchart 
as shown in 
Fig.~\ref{figNoMinimalChartSix}(c).
We use notations in 
Fig.~\ref{figNoMinimalChartSix}(c), 
here
$e_4'$ is the terminal edge 
of label $m+2$ containing $w_4$, 
and 
 $e_2,e_3, e', e'', e'''$ are 
the internal edges of label ${m+1}$ with 
$\partial e_2=\{ w_2,w_5\}$, 
$\partial e_3=\{ w_3,w_6\}$, 
$\partial e'=\{ w_2,w_4\}$, 
$\partial e''=\{ w_3,w_4\}$ and 
$\partial e'''=\{ w_2,w_3\}$. 
Further $\widetilde{e}_1,\widetilde{e}_2,\widetilde{e}_3$
are the internal edges of label ${m+2}$
such that $\partial \widetilde{e}_1=\{ w_4,w_5\}$,
$\partial \widetilde{e}_2=\{ w_4,w_6\}$, 
and $\widetilde{e}_3$ is oriented from $w_6$ to $w_5$.

Let $E$ be the 3-angled disk of $\Gamma_{m+2}$
bounded by $\widetilde{e}_1,\widetilde{e}_2,\widetilde{e}_3$ with $E\ni w_2,w_3$. 
Set $X=e'\cup e''\cup e'''\cup e_2\cup e_3\cup e_4'\cup \partial E$.

Now 
$(\Gamma_{m+1}\cup\Gamma_{m+2}\cup\Gamma_{m+3})\cap E=X$. 
For, by applying Lemma~\ref{MM+2Edge} to 
$e_2,e_3$ respectively,
we can assume that 
\begin{enumerate}
\item[(1)] 
neither $e_2$ nor $e_3$ contains a crossing.
\end{enumerate}
Now $\partial E\subset\Gamma_{m+2}$
implies that 
the closure of 
$(\Gamma_{m+1}\cup\Gamma_{m+2}\cup\Gamma_{m+3})\cap E-X$ 
consists of hoops, and rings.
In the disk $E$ 
each of the hoops and the rings is simple
by Statement (1). 
Hence by Assumption~\ref{NoSimpleHoop} 
and Assumption~\ref{AssumptionRing}, 
we have that 
$(\Gamma_{m+1}\cup\Gamma_{m+2}\cup\Gamma_{m+3})
\cap E=X$.

Applying C-II moves,
we can shift the black vertex 
of the terminal edge $e'_4$ 
near the edge $e'''$.
Applying a C-III move,
we obtain the chart as shown 
in Fig.~\ref{figNoMinimalChartSixProof}(b).

Let $F$ be the 4-angled disk of $\Gamma_{m+1}$ 
as shown in Fig.~\ref{figNoMinimalChartSixProof}(b).
Then $F$ is an M4-disk of label $m+1$.
By Lemma~\ref{M4-diskLemma},
we can apply a C-I-M4 move for $F$ 
so that 
we obtain the chart as shown in 
Fig.~\ref{figNoMinimalChartSixProof}(c).
The terminal edge of label $m+1$ containing $w_4$
is not middle at $w_4$.
Thus by a C-III move,
we obtain a minimal chart 
with six white vertices and 
a lens.
This contradicts Lemma~\ref{PolandLemma}(b).
Therefore $\Gamma$ is not minimal.
\end{proof}

\begin{figure}[ht]
\centerline{\includegraphics{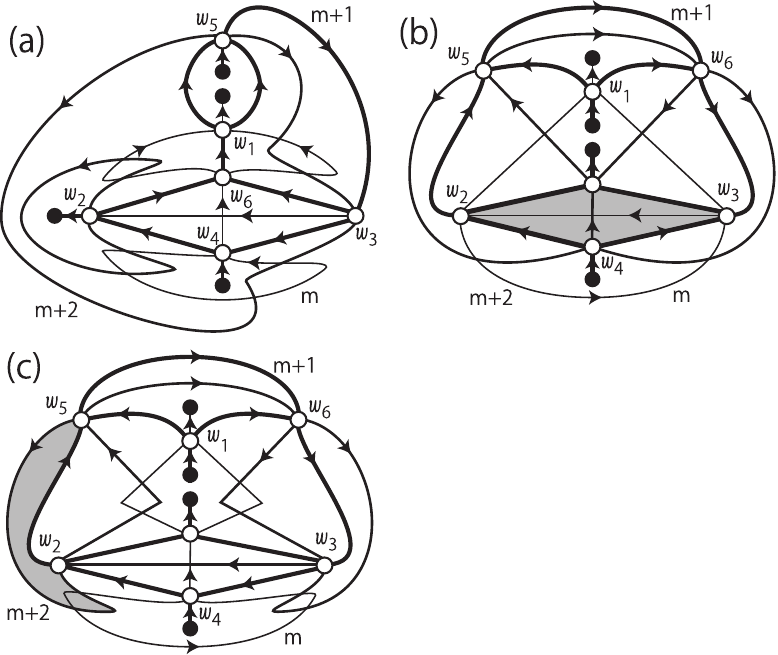}}
\vspace{5mm}
\caption{\label{figNoMinimalChartSixProof}
The thick lines are edges of label $m+1$.
(b) The gray region is the 4-angled disk $F$.
(c) The gray region is a lens of type $(m+1,m+2)$.}
\end{figure}

\begin{lemma}
\LABEL{MinimalChartSix}
Let $\Gamma$ be a minimal chart with $w(\Gamma)=6$.
If $\Gamma$ contains the subchart as shown in 
Fig.~\ref{figNoMinimalChartSix}{\rm (d)},
then $\Gamma$ is C-move equivalent to a minimal chart containing a subchart in the lor-family of 
the $2$-twist spun trefoil.
\end{lemma}

\begin{proof}
Let $\Gamma$ be a minimal chart with $w(\Gamma)=6$
containing the subchart as shown in 
Fig.~\ref{figNoMinimalChartSix}{\rm (d)}.
We use the notations in 
Fig.~\ref{figNoMinimalChartSix}{\rm (d)}, 
here $e^*,e_4$ are 
the terminal edges of label $m+1$
containing $w_3,w_4$ respectively, 
$e_1$ is a terminal edge of label $m$ 
containing $w_1$,
 $e_4',e_4''$ are 
the internal edges of label ${m+2}$ 
with $\partial e_4'=\{ w_4,w_5\}$ and 
$\partial e_4''=\{ w_4,w_6\}$, 
 $e$ is 
the internal edge of label $m$ 
oriented from $w_3$ to $w_2$, and 
$e_3'$ is 
the internal edge of label $m$ 
with $\partial e_3'=\{ w_1,w_3\}$.
Let $E$ be 
the 6-angled disk of $\Gamma_{m+1}$ 
with $w(\Gamma\cap{\rm Int}E)=0$ and 
$E\supset e^*\cup e_4$ but $E\not\supset e_1$.

Set 
$X=
(e\cup e_3'\cup e^*\cup e_4\cup e_4'\cup e_4'')
\cup\partial E$.
Since $\partial E\supset \Gamma_{m+1}$, 
and 
since Int$E$ contains no white vertex, 
each connected component of 
$Cl((\Gamma_m\cup\Gamma_{m+1}\cup\Gamma_{m+2})\cap E-X)$ 
is a simple hoop or a simple ring.
By Assumption~\ref{NoSimpleHoop} and 
Assumption~\ref{AssumptionRing},
we have that 
$(\Gamma_m\cup\Gamma_{m+1}\cup\Gamma_{m+2})\cap E=X$.

By C-II moves,
 we move the two black vertices in $e^*,e_4$
near the crossing $e\cap e_4'$.
By a C-I-M1 move,
we create a hoop of label $m+1$ near 
the crossing $e\cap e_4'$ 
(see Fig.~\ref{figMinimalChartSixProof}(a)).
Then we apply two C-III moves for the two terminal edges $e^*,e_4$,
and then we obtain the chart as shown in Fig.~\ref{figMinimalChartSixProof}(b).
By Lemma~\ref{M4-diskLemma},
we can apply a C-I-M4 move for the 4-angled disk of $\Gamma_{m+1}$ so that
we obtain the chart as shown in Fig.~\ref{figMinimalChartSixProof}(c).

Let $E'$ be the 2-angled disk of $\Gamma_m$ with
$w(\Gamma\cap{\rm Int}E')=0$
and $\partial E'\ni w_2,w_4$.
Then $E'$ is an M3-disk of type $(m,m+1)$.
Thus 
by Lemma~\ref{M3-diskLemma}
we can apply a C-I-M3 move for $E'$ and then
we obtain the chart with six white vertices 
as shown in Fig.~\ref{figMinimalChartSixProof}(d).
This chart contains a subchart in the lor-family of the 2-twist spun trefoil.
\end{proof}

\begin{figure}[ht]
\centerline{\includegraphics{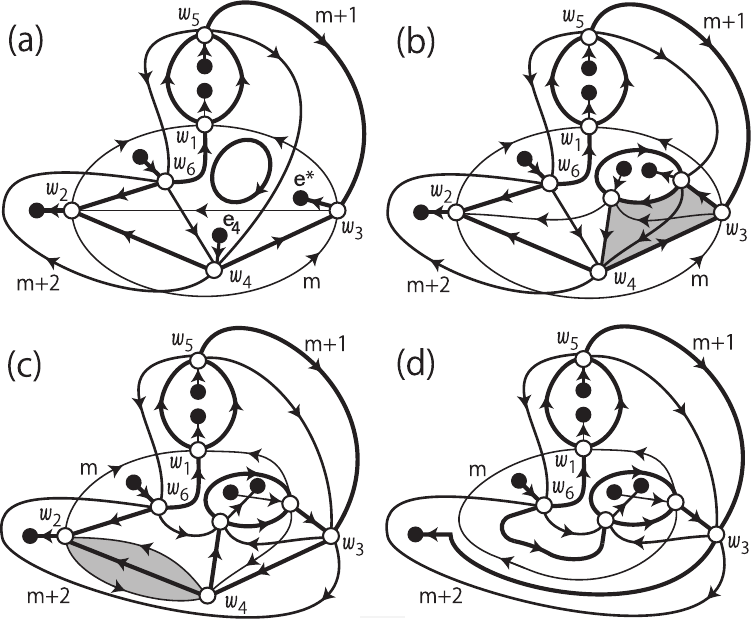}}
\vspace{5mm}
\caption{\label{figMinimalChartSixProof}
The thick lines are edges of label $m+1$.
(b) The gray region is the 4-angled disk of $\Gamma_{m+1}$.
(c) The gray region is an M3-disk.}
\end{figure}


\section{{\large Proof of Main Theorem}}
\label{s:MainTheorem}

In this section,
we assume that 
\begin{enumerate}
\item[(1)] $\Gamma$ is a minimal chart of type $(m;3,3)$.
\end{enumerate}
By Lemma~\ref{LoopSixVertices},
we can assume that 
\begin{enumerate}
\item[(2)] $\Gamma$ does not contain any loop.
\end{enumerate}

{\bf Claim 1.} The subgraph $\Gamma_m$ contains a graph $G$ as shown in Fig.~\ref{figThree}.

{\it Proof of Claim $1$.}
First of all,
we shall show $w(\Gamma_m)=3$.
Since $\Gamma$ is of type $(m;3,3)$,
we have that $w(\Gamma_m\cap\Gamma_{m+1})=3$ and 
$w(\Gamma_{m-1})=0$ 
by the conditions (i) and (ii)
of type of a chart.
Since any white vertex in $\Gamma_m$ is contained in 
$\Gamma_{m-1}$ or $\Gamma_{m+1}$,
we have $w(\Gamma_m)=3$.

Let $G$ be a connected component of $\Gamma_m$
with $w(G)\ge1$.
We shall show $w(G)=3$.
By Lemma~\ref{TwoThreeLoop}(a),
we have $w(G)\ge2$.
If $w(G)=2$,
then there exists a white vertex $w$ in $\Gamma_m$
with $w\not\in G$.
Since $w(G)=2$ and $w(\Gamma_m)=3$,
the connected component of $\Gamma_m$
containing $w$ contains exactly one white vertex.
This contradicts Lemma~\ref{TwoThreeLoop}(a).
Hence $w(G)=3$.
By Lemma~\ref{TwoThreeLoop}(b),
the graph $G$ is the graph as shown in Fig.~\ref{figThree}.
Thus $\Gamma_m$ contains the graph $G$ as shown in Fig.~\ref{figThree}.
\hfill {$\square$}\vspace{1.5em}

The closures of connected components of $S^2-G$ 
consists of three disks.
Let $e_1$ be the terminal edge of label $m$.
Let $D_1$ be one of the three disks
with $D_1\cap e_1=\emptyset$,
$D_2$ one of the three disks
with $D_2\supset e_1$,
and $D_3$ the last one of the three disks 
(see Fig.~\ref{Type33}(a)).
Then $D_1$ is a 2-angled disk without feelers,
$D_2$ is a 3-angled disk with one feeler,
and $D_3$ is a 3-angled disk without feelers.
Let $w_1,w_2,w_3$ be white vertices in $\Gamma_m$
with $w_1\in e_1$,
and $e$ the internal edge of label $m$ with $e=D_1\cap D_3$.
In this section
we assume that 
\begin{enumerate}
\item[(3)] $\Gamma$ is locally minimal with respect to $D_2$,
\item[(4)] the terminal edge $e_1$ of label $m$ is oriented from $w_1$,
\item[(5)] the edge $e$ is oriented from $w_3$ to $w_2$,
\item[(6)] the point at infinity $\infty$ is contained in Int$D_2$.
\end{enumerate}
We use the notations shown in Fig.~\ref{Type33}(a).
\vspace{1.5em}

\begin{figure}[ht]
\centerline{\includegraphics{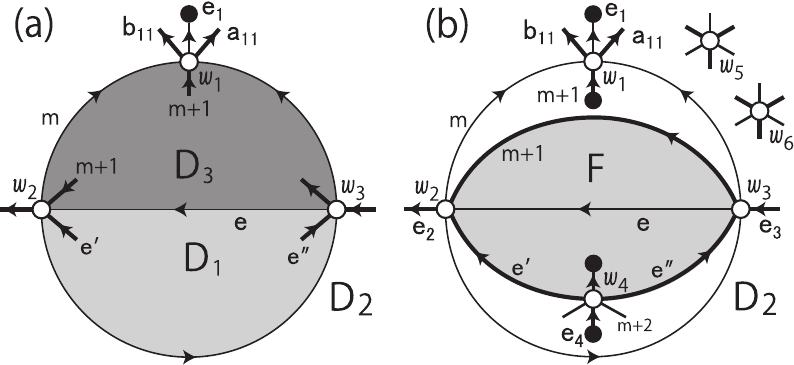}}
\vspace{5mm}
\caption{\label{Type33}
The thick lines are edges of label $m+1$.
(a) The light gray region is the disk $D_1$,
and the dark gray region is the disk $D_3$.
(b) The gray region is the disk $F$.}
\end{figure}

{\bf Claim 2.} 
$w(\Gamma\cap{\rm Int}D_1)\ge1$ and 
$w(\Gamma\cap{\rm Int}D_2)\ge1$.

{\it Proof of Claim $2$.}
Since $\partial D_1$ is oriented anticlockwise,
by Lemma~\ref{TwoAngledDisk}(a)
we have $w(\Gamma\cap{\rm Int}D_1)\ge1$.

Since $D_2$ is a special $3$-angled disk with one feeler,
by Lemma~\ref{Theorem3AngledDisk}
we have $w(\Gamma\cap{\rm Int}D_2)\ge1$.
\hfill {$\square$}\vspace{1.5em}

Let $w_4,w_5,w_6$ be white vertices in $\Gamma_{m+2}$.
By Claim $2$,
without loss of generality
we can assume that 
$w_4\in {\rm Int}D_1$ and
$w_5\in {\rm Int}D_2$. 
\vspace{1.5em}

{\bf Claim 3.}
 $w_6\not\in {\rm Int}D_1$.

{\it Proof of Claim $3$.}
Let $e',e''$ be the edges of label ${m+1}$ in $D_1$ 
containing $w_2,w_3$ respectively.
Then both of $e',e''$ are oriented inward at white vertices on the boundary.
Suppose $w_6\in {\rm Int}D_1$.
Then $w(\Gamma\cap {\rm Int}D_1)=2$
and $w(\Gamma_{m+1}\cap {\rm Int}D_1)=2$.
Thus by Lemma~\ref{TwoAngledDisk}(c),
the disk $D_1$ contains an element of RO-families 
of the two pseudo charts 
as shown in Fig.~\ref{fig2AngledDisk}(c) and (d). 
In Fig.~\ref{fig2AngledDisk}(c) and (d),
there exists an internal edge of label ${m+1}$ 
oriented outward at the white vertex on the boundary.
Hence $e'$ or $e''$ is oriented outward at the white vertex on the boundary $\partial D_1$.
This is a contradiction. 
Hence $w_6\not\in {\rm Int}D_1$.
\hfill {$\square$}\vspace{1.5em}

{\bf Claim 4.}
 $w_6\not\in {\rm Int}D_2$.

{\it Proof of Claim $4$.}
Suppose $w_6\in {\rm Int}D_2$.
Then 
$w(\Gamma\cap{\rm Int}D_1)=1$, 
$w(\Gamma\cap{\rm Int}D_2)=2$ and
$w(\Gamma\cap{\rm Int}D_3)=0$.
By Lemma~\ref{TwoAngledDisk}(b),
the disk $D_1$ contains pseudo chart as shown in 
Fig.~\ref{fig2AngledDisk}(b).
By Lemma~\ref{3AngledDiskWithoutFeeler},
the disk $D_3$ contains pseudo chart as shown in 
Fig.~\ref{fig3AngledDisk}(b).
Thus we have the pseudo chart as shown in Fig.~\ref{figIOC}(a).

Let $e_4$ be the terminal edge of label $m+1$
containing $w_4$.
Let $F$ be the 3-angled disk of $\Gamma_{m+1}$
containing $e$
with $\partial F\ni w_2,w_3,w_4$.
Then $w(\Gamma\cap{\rm Int}F)=0$. 
Now we have the same situation 
in the example of 
IO-Calculation 
in Section~\ref{s:IOC}. 
Thus 
we have $e_4\not\subset F$ 
by IO-Calculation 
with respect to $\Gamma_{m+2}$ in $F$
(see Statement (3) in Section~\ref{s:IOC}).
Hence we have the pseudo chart in 
Fig.~\ref{Type33}(b).
We use the notations shown in Fig.~\ref{Type33}(b).

Now $e_2$ is an internal edge (or a terminal edge) of label ${m+1}$ in $D_2$
containing $w_2$.
Then $e_2$ is not a terminal edge.
For,
if $e_2$ is a terminal edge,
then the 3-angled disk $F$ satisfies the condition in 
Lemma~\ref{LemmaTriangle}.
Thus $w(\Gamma\cap{\rm Int}F)\ge1$.
This contradicts the fact that 
$w(\Gamma\cap{\rm Int}F)=0$.
Hence 
 $e_2$ is not a terminal edge.

Let $F'=Cl(S^2-F)$.
Then $F'$ is a 3-angled disk of $\Gamma_{m+1}$.
We shall show $e_2\not\ni w_3$.
For, if $e_2\ni w_3$,
then the edge $e_2$ separates $F'$ into two disks.
One of the two disks contains the terminal edge $e_4$,
say $D$.
Let $D'$ be the other disk.
By IO-Calculation with respect to $\Gamma_{m+2}$ in $D$,
we have that $D$ contains $w_5$ or $w_6$.
Further,
by IO-Calculation with respect to $\Gamma_{m+1}$ in $D'$,
we have that $D'$ contains $w_5$ or $w_6$.
Without loss of generality
we can assume that $w_5\in D$ and $w_6\in D'$.
Thus $w(\Gamma_{m+2}\cap D')=1$
and there exists a connected component of $\Gamma_{m+2}$
containing exactly one white vertex $w_6$.
This contradicts Lemma~\ref{TwoThreeLoop}(a).
Hence $e_2\not\ni w_3$.
Since $e_2\not=a_{11}$ and $e_2\not=b_{11}$
and since $e_2$ is not a terminal edge, 
we have  $e_2\ni w_5$ or $e_2\ni w_6$.
Without loss of generality
we can assume
$e_2\ni w_5$.

Now $e_3$ is  an internal edge (or a terminal edge) of label ${m+1}$ in $D_2$ 
containing $w_3$.
Since the edge $e_3$ is not middle at $w_3$,
by Assumption~\ref{NoTerminal}
the edge $e_3$ is not a terminal edge.
Hence there are four cases:
(1) $e_3=a_{11}$,
(2) $e_3=b_{11}$,
(3) $e_3\ni w_5$,
(4) $e_3\ni w_6$.

For {\bf Case (1)},
there exists a lens of type $(m,m+1)$ whose boundary contains $e_3$.
This contradicts Lemma~\ref{PolandLemma}(b).
Hence Case (1) does not occur.

For {\bf Case (2)},
we have $a_{11}\ni w_6$
because $a_{11}$ is not middle at $w_1$.
Hence there exists a loop of label $m+1$ containing $w_6$.
This contradicts the fact that 
$\Gamma$ does not contain any loop
(see Statement (2) in this section).
Hence Case (2) does not occur.

For {\bf Case (3)},
let $e_5$ be an internal edge 
(or a terminal edge) of label ${m+1}$ 
containing $w_5$
different from $e_2$ and $e_3$.
Since $e_2$ is oriented from $w_2$ to $w_5$
and since $e_3$ is oriented from $w_5$ to $w_3$,
the edge $e_5$ is not middle at $w_5$.
By Assumption~\ref{NoTerminal},
the edge $e_5$ is not a terminal edge.

Let $E$ be the 4-angled disk of $\Gamma_{m+1}$
with $\partial E=e'\cup e''\cup e_2\cup e_3$
and $E\ni w_1$.
If $e_5\not\subset E$,
then $w_6\in e_5$, because $e_5$ is not a terminal edge.
Thus there exists a loop of label $m+1$
containing $w_6$.
This contradicts the fact that 
$\Gamma$ does not contain any loop.
Hence $e_5\subset E$ (see Fig.~\ref{figClaim4}(a)).

Let $e'''$ be the internal edge of label ${m+1}$
with $\partial e'''=\{w_2,w_3\}$.
Let $E'$ be the 3-angled disk of $\Gamma_{m+1}$ in $E$
with $\partial E'=e_2\cup e_3\cup e'''$.
By IO-Calculation with respect to $\Gamma_{m+1}$ 
in $E'$,
we have $w_6\in E'\subset E$.
Thus $w(\Gamma\cap(S^2-E))=0$.
However,
by IO-Calculation with respect to $\Gamma_{m+2}$ in $Cl(S^2-E)$,
there exists at least one white vertex of $\Gamma_{m+2}$ in $S^2-E$.
This is a contradiction.
Hence Case (3) does not occur.

For {\bf Case (4)},
let $e_2',e_2''$ be internal edges (or terminal edges) of label ${m+1}$ containing $w_5$ 
different from $e_2$ such that 
$e_2,e_2',e_2''$ lie clockwise around $w_5$ in this order.
Let $e_3',e_3''$ be internal edges (or terminal edges) of label ${m+1}$ 
containing $w_6$
different from $e_3$ such that 
$e_3,e_3',e_3''$ lie clockwise around $w_6$ in this order
(see Fig.~\ref{figClaim4}(b)).

We can show that $a_{11}=e_3'$.
If not,
then there are three cases:
(4-1) $a_{11}=e_2'$, 
(4-2) $a_{11}=e_2''$, 
(4-3) $a_{11}=e_3''$.

For {\bf Case (4-1)},
there exists a loop of label $m+1$ containing $w_6$.
This contradicts the fact that 
$\Gamma$ does not contain any loop.

For {\bf Case (4-2)},
the edge $b_{11}$ is a terminal edge not middle at $w_1$.
This contradicts Assumption~\ref{NoTerminal}.

For {\bf Case (4-3)}, 
the edge $e_{3}'$ is a terminal edge not middle at $w_6$.
This contradicts Assumption~\ref{NoTerminal}.

Hence we have $a_{11}=e_3'$.
Similarly we can show that $b_{11}=e_2''$.
Since the edge $e_3''$ is not middle at $w_6$,
we have $e_3''=e_2'$.
Hence by the help of New Disk Lemma (Lemma~\ref{NewDiskLemma}),
the chart $\Gamma$ contains the subchart as shown in 
Fig.~\ref{figNoMinimalChartSix}(c).
By Lemma~\ref{NoMinimalChartSix},
the chart $\Gamma$ is minimal.
This is a contradiction.
Hence Case (4) does not occur.
Therefore we complete the proof of Claim~$4$.
\hfill {$\square$}\vspace{1.5em}

\begin{figure}[htb]
\centerline{\includegraphics{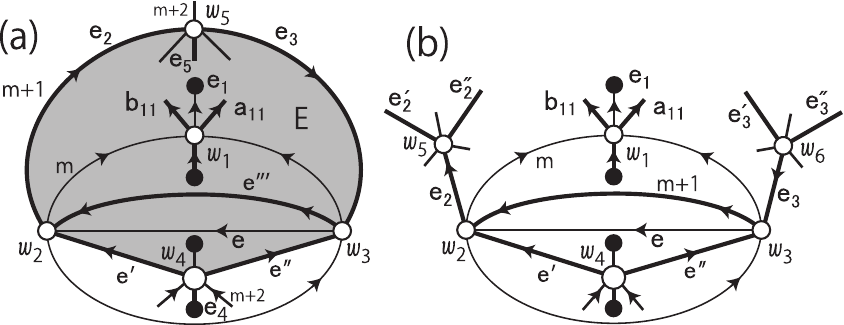}}
\vspace{5mm}
\caption{\label{figClaim4}
The thick lines are edges of label $m+1$.
(a) The gray region is the disk $E$.}
\end{figure}


{\it Proof of Main Theorem.}
Now we start from the pseudo chart as shown in Fig.~\ref{Type33}(a).
By Claim~$2$, Claim~$3$ and Claim~$4$,
we have
 $w(\Gamma\cap{\rm Int}D_1)=1$,
$w(\Gamma\cap{\rm Int}D_2)=1$ and
$w(\Gamma\cap{\rm Int}D_3)=1$.
Without loss of generality,
we can assume that 
Int$D_1$ contains a white vertex $w_4$,
Int$D_2$ contains a white vertex $w_5$,
Int$D_3$ contains a white vertex $w_6$.

By Lemma~\ref{TwoAngledDisk}(b),
the disk $D_1$ contains the pseudo chart as shown in 
Fig.~\ref{fig2AngledDisk}(b).
Since $D_2$ is a special 3-angled disk with one feeler,
by Lemma~\ref{Theorem3AngledDisk}
the disk $D_2$ contains
an element of RO-families of the two pseudo charts
as shown in Fig.~\ref{fig3AngledDisk}(g) and (h)
(see Fig.~\ref{figMainTheorem}(a) and (b)).

\begin{figure}[htb]
\centerline{\includegraphics{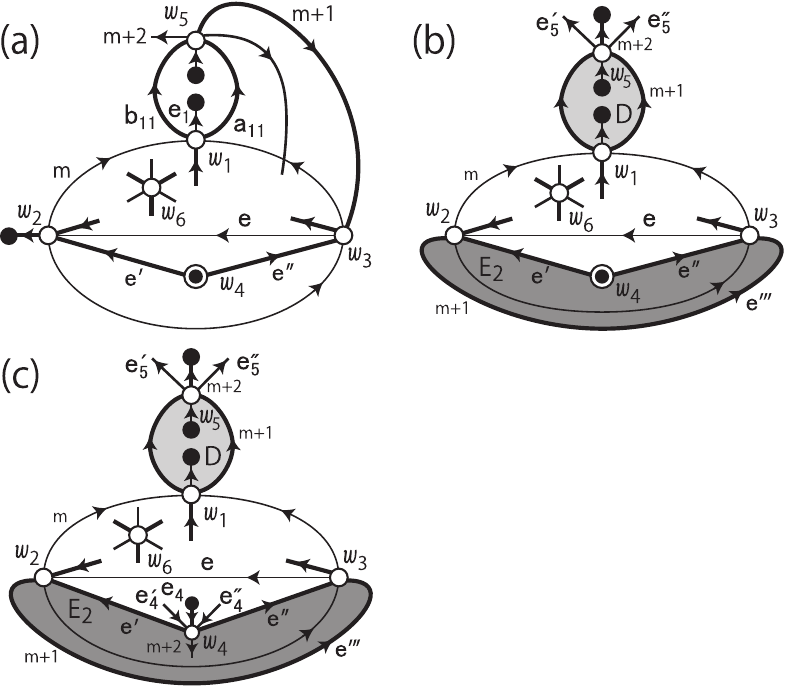}}
\vspace{5mm}
\caption{\label{figMainTheorem}
The thick lines are edges of label $m+1$.
(b),(c) The dark gray region is the 3-angled disk $E_2$,
and the light gray region is the 2-angled disk $D$.}
\end{figure}

Suppose that $\Gamma$ contains the pseudo chart as shown in Fig.~\ref{figMainTheorem}(b).
We use the notations shown in Fig.~\ref{figMainTheorem}(b),
here $e$ is the internal edge of label $m$ 
oriented from $w_3$ to $w_2$, and
$e',e'',e'''$ are internal edges of label ${m+1}$
with $\partial e'=\{w_2,w_4\}$,
$\partial e''=\{w_3,w_4\}$ and
$\partial e'''=\{w_2,w_3\}$.
Let $E_1,E_2$ be the 3-angled disks of $\Gamma_{m+1}$
with $E_1\cap E_2=\partial E_1=\partial E_2=e'\cup e''\cup e'''$
and $E_1\supset e$.
Let $e_4$ be the terminal edge of label $m+1$ containing $w_4$.
By IO-Calculation with respect to $\Gamma_{m+2}$ in $E_2$,
we have $e_4\not\subset E_2$.
Thus $e_4\subset E_1$ (see Fig.~\ref{figMainTheorem}(c)).

Let $D$ be the 2-angled disk of $\Gamma_{m+1}$ in $D_2$
with $\partial D\ni w_1,w_5$ 
(see Fig.~\ref{figMainTheorem}(c)).
Let $n_I$ be the number of inward arcs of label $m+2$ in $Cl(E_1-D)$ and
$n_O$ the number of outward arcs of label $m+2$ in $Cl(E_1-D)$.
We shall count the numbers $n_I$ and $n_O$ as follows.

{\bf Assertion 1.} There are two outward arcs of label $m+2$ 
containing $w_5$ in $Cl(E_1-D)$.

For, the two edges $e_5',e_5''$ are internal edges (or terminal edges) of label ${m+2}$ containing $w_5$ in $Cl(E_1-D)$ 
(see Fig.~\ref{figMainTheorem}(c)).
Since the two edges are oriented outward at $w_5$,
there are two outward arcs of label $m+2$ 
containing $w_5$ in $Cl(E_1-D)$.

{\bf Assertion 2.} There are two inward arcs of label $m+2$ containing $w_4$ in $Cl(E_1-D)$.

For, $e_4\subset E_1$ implies that
 there are two internal edges (or terminal edges)
 of label ${m+2}$ in $Cl(E_1-D)$  
containing $w_4$, say
$e_4',e_4''$ (see Fig.~\ref{figMainTheorem}(c)).
Since the terminal edge $e_4$ is oriented inward at $w_4$,
the two edges  $e_4',e_4''$ are oriented inward at $w_4$.
Thus there are two inward arcs of label $m+2$ 
containing $w_4$ in $Cl(E_1-D)$.

Now for the white vertex $w_6\in D_3\subset Cl(E_1-D)$,
there are two cases:
\begin{enumerate}
\item[(b-1)] there is one inward arc of label $m+2$ 
containing $w_6$
and there are two outward arcs of label $m+2$ 
containing $w_6$,  
\item[(b-2)] there is one outward arc of label $m+2$ 
containing $w_6$
and there are two inward arcs of label $m+2$ 
containing $w_6$.
\end{enumerate}

For {\bf Case (b-1)},
by Assumption~\ref{NoTerminal}
the white vertex $w_6$ is contained in at most one terminal edge of label $m+2$.
Let $e_6$ be an internal edge (or a terminal edge) of label ${m+2}$ middle at $w_6$.

Suppose that $e_6$ is a terminal edge.
Then by Condition (b-1)
the terminal edge $e_6$ is oriented inward at $w_6$.
Thus 
\begin{enumerate}
\item[(*)] there exists an outward arc of label $m+2$ containing the black vertex in the terminal edge $e_6$.
\end{enumerate}
Since none of $e_4',e_4'',e_5',e_5''$ contain middle arcs at $w_4$ or $w_5$,
none of $e_4',e_4'',e_5',e_5''$ are terminal edges
by Assumption~\ref{NoTerminal}.
Thus the two edges $e_4',e_4''$ contain $w_5$ or $w_6$,
and the two edges $e_5',e_5''$ contain $w_4$ or $w_6$.
Hence
by (*) and Assertion~$1$ and Assertion~$2$
we have $n_I=3$ and $n_O=5$.
This is a contradiction
by IO-Calculation with respect to $\Gamma_{m+2}$ in $Cl(E_1-D)$.

Similarly
for the case that $e_6$ is not a terminal edge,
we have $n_I=3$ and $n_O=4$.
This is a contradiction.
Hence Case (b-1) does not occur.

For {\bf Case (b-2)},
in a similar way to Case (b-1), 
we have  $n_I=4$ or $5$ and $n_O=3$.
This is a contradiction by IO-Calculation with respect to $\Gamma_{m+2}$ in $Cl(E_1-D)$.
Hence Case (b-2) does not occur.
Thus $\Gamma$ does not contain the pseudo chart as shown in Fig.~\ref{figMainTheorem}(b).

Suppose that $\Gamma$ contains the pseudo chart as shown in Fig.~\ref{figMainTheorem}(a).
Since the disk $D_3$ is a 3-angled disk without feelers 
with $w(\Gamma\cap{\rm Int}D_3)=1$,
by Lemma~\ref{Theorem3AngledDisk}
there are three cases:
\begin{enumerate}
\item[(a-1)] $D_3$ contains an element of the RO-family of the pseudo chart as shown in Fig.~\ref{fig3AngledDisk}(c)
(see Fig.~\ref{figMainTheoremCaseA}(a)),
\item[(a-2)] $D_3$ contains an element of RO-family of the pseudo chart as shown in Fig.~\ref{fig3AngledDisk}(d)
(see Fig.~\ref{figMainTheoremCaseA}(b)),
\item[(a-3)] $D_3$ contains an element of RO-family of the pseudo chart as shown in Fig.~\ref{fig3AngledDisk}(e)
(see Fig.~\ref{figMainTheoremCaseA}(c)).
\end{enumerate}
Let $e_4$ be the terminal edge of label $m+1$ containing $w_4$.

For {\bf Case (a-1)},
let $E$ be the 4-angled disk of $\Gamma_{m+1}$ with 
$\partial E\ni w_2,w_3,w_4,w_6$ and $E\supset e$.
By IO-Calculation with respect to $\Gamma_{m+2}$ in $E$,
we have $e_4\not\subset E$.
Thus by the help of New Disk Lemma (Lemma~\ref{NewDiskLemma}),
the chart $\Gamma$ contains the subchart as shown in 
Fig.~\ref{figNoMinimalChartSix}(a).
By Lemma~\ref{NoMinimalChartSix},
the chart $\Gamma$ is not minimal.
This is a contradiction.
Hence Case (a-1) does not occur.

For {\bf Case (a-2)},
let $E$ be the 6-angled disk of $\Gamma_{m+1}$ with 
$E\not\supset e_1$ and $E\supset e$.
By IO-Calculation with respect to $\Gamma_{m+2}$ in $E$,
we have $e_4\subset E$.
Thus by the help of New Disk Lemma (Lemma~\ref{NewDiskLemma}),
the chart $\Gamma$ contains the subchart as shown in 
Fig.~\ref{figNoMinimalChartSix}(d).
By Lemma~\ref{MinimalChartSix},
the chart $\Gamma$ is C-move equivalent to a minimal chart containing a subchart in the lor-family of the 2-twist spun trefoil as shown in Fig.~\ref{figTrefoil}.

For {\bf Case (a-3)},
let $E$ be the 6-angled disk of $\Gamma_{m+1}$ with 
$E\not\supset e_1$ and $E\supset e$.
By IO-Calculation with respect to $\Gamma_{m+2}$ in $E$,
we have $e_4\not\subset E$.
Thus by the help of New Disk Lemma (Lemma~\ref{NewDiskLemma}),
the chart $\Gamma$ contains the subchart as shown in 
Fig.~\ref{figNoMinimalChartSix}(b).
By Lemma~\ref{NoMinimalChartSix},
the chart $\Gamma$ is not minimal.
This is a contradiction.
Hence Case (a-3) does not occur.

Therefore we complete the proof of Main Theorem.
\hfill {$\square$}\vspace{1.5em}

\begin{figure}[htb]
\centerline{\includegraphics{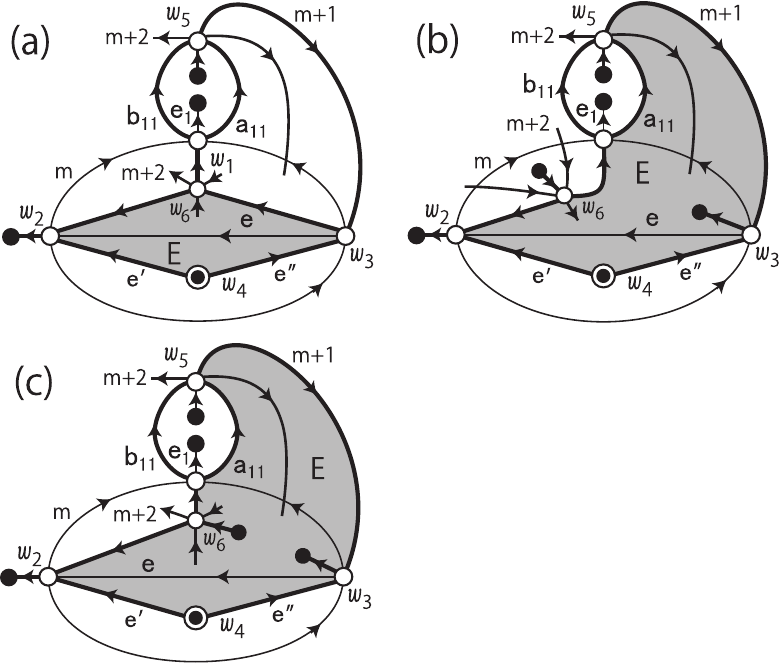}}
\vspace{5mm}
\caption{\label{figMainTheoremCaseA}
The thick lines are edges of label $m+1$
and the gray region is the disk $E$.}
\end{figure}


\vspace{5mm}

\begin{minipage}{65mm}
{Teruo NAGASE
\\
{\small Tokai University \\
4-1-1 Kitakaname, Hiratuka \\
Kanagawa, 259-1292 Japan\\
\\
nagase@keyaki.cc.u-tokai.ac.jp
}}
\end{minipage}
\begin{minipage}{65mm}
{Akiko SHIMA 
\\
{\small Department of Mathematics, 
\\
Tokai University
\\
4-1-1 Kitakaname, Hiratuka \\
Kanagawa, 259-1292 Japan\\
shima@keyaki.cc.u-tokai.ac.jp
}}
\end{minipage}

\end{document}